\documentclass[11pt]{article}
\usepackage{latexsym}
\usepackage{mathrsfs}
\usepackage{amsmath,amsthm}
\usepackage{graphicx}
\usepackage{amssymb}
\usepackage{CJK}
\usepackage{color}
\newtheorem{thm}{Theorem}[section]
\newtheorem{cor}[thm]{Corollary}
\newtheorem{lem}[thm]{Lemma}

\newtheorem{rem}[thm]{Remark}

\newtheorem{exa}{Example}

\newtheorem{alg}[thm]{Algorithm}
\newtheorem{obs}{Observation}
\DeclareGraphicsExtensions{.eps,.eps.gz}
\normalsize \rm
\usepackage[centerlast]{caption}

\title{On the maximal matchings of trees}
\author{Lingjuan Shi$^{a}$\footnote{
The first author is supported by the Natural Science Foundation of Shaanxi Province (grant no. 2024JC-YBQN-0053) and by NSFC (grant no. 11901458 and 12161002).},
Wei Li$^{b,c}$\footnote{The author is supported by Guangdong Basic and Applied Basic Research Foundation in China (grant no. 2022A1515012342) and the Natural Science Basic Research Plan in Shaanxi Province of China (grant no. 2020JM-133).},
Kai Deng$^d$ \footnote{Corresponding author. The author is supported by NSFC (grant no. 12161002) and the Natural Science Foundation of Ningxia (grant no. 2022AAC03285)},
\date{\small $^{a}$ School of Software, Northwestern Polytechnical University, Xi'an, Shaanxi 710072, P. R. China\\
$^{b}$ Research \& Development Institute of Northwestern Polytechnical University in Shenzhen, Shenzhen, Guangdong 518057, China\\
$^{c}$ School of mathematics and statistics, Northwestern Polytechnical University, Xi'an, Shaanxi 710072, P. R. China\\
$^{d}$ School of Mathematics and Information Science, North Minzu University, Yinchuan, Ningxia 750027, P. R. China\\
E-mails: shilj18@nwpu.edu.cn, liw@nwpu.edu.cn, dengkai04@126.com}
}

\begin{document}
\maketitle
\begin{abstract}
An independent edge set of graph $G$ is a matching, and
is maximal if it is not a proper subset of any other matching of $G$.
The number of all the maximal matchings of $G$ is denoted by $\Psi(G)$.
In this paper, an algorithm to count $\Psi(T)$ for a tree $T$ is given.
We show that for any tree $T$ with $n$ vertices, $\Psi(T)\geq\lceil\frac{n}{2}\rceil$, and the tree which obtained the lower bound is characterized.


\end{abstract}
\baselineskip=0.3in

\section{Introduction and Preliminaries}
Let $G$ be a finite and simple graph with vertex set $V(G)$ and edge set $E(G)$. The number of vertices of $G$ is called its order, and the number of edges of $G$ is called its size.
The set of neighbors of a vertex $v$ in graph $G$ is
denoted by $N_G(v)$. The degree of a vertex $v$ is the cardinality of $N_G(v)$, denoted by $d_G(v)$.
A vertex is called a \emph{leaf} if its degree is one. A vertex with degree at least three is called a
\emph{branch vertex}.
Let $G$ be a connected graph which is not isomorphic to complete bipartite graph $K_{1,t}$. A \emph{pendant star} with $t$ rays ($t\geq3$) in $G$ is a subgraph $K_{1,t}$ with $t-1$ leaves of $G$ and the central vertex of degree $t$ both in the star and in $G$.
A \emph{pendent path} $P$ of $G$ is a path whose one end vertex is a branch vertex, one end vertex is a leaf and the intermedient vertex (if exists) have degree two in $G$.

For graph $G$, an edge set $M\subseteq E(G)$ is a \emph{matching} of $G$ if any two elements in $M$ are not incident with a common vertex.
And $M$ is \emph{maximal} if it is not contained in any other matching of $G$, $M$ is \emph{maximum} if its size is maximum among all matchings.
We use $\Psi(G)$ to denote the number of maximal matchings
of $G$.

In paper \cite{Rooted-Plane-Trees}, Klazar investigate several counting examples for rooted plane trees, including
the average numbers of maximal matchings of a rooted plane tree on $n$ vertices.
After that Wagner \cite{tree-matching-number} extended the results of Klazar to formulas for matchings, maximal matchings and maximum matchings for three types of simply
generated trees.
For small number $n$ of vertices,
$n\leq19$, J. G\'{o}rska and Z. Skupie\'{n} \cite{Tree-maximal-macthing} characterized the trees which have maximum number of maximal matchings over all trees with $n$ vertices. And  they have found exponential upper and lower bounds on the maximum number of maximal matchings among $n$-vertex trees.

Recently, there have been some results for the
maximal matchings in $3$-connected planar graphs
and graphs with bounded maximum degree \cite{Tight-maximal-maximum},
linear polymers \cite{Doslic, maximal-matching-linear-polymers}, benzenoids \cite{saturation-ben, shi-deng-linear}, fullerenes \cite{saturation-fulle2015, saturation-fulle2008}, nanocones and
 nanotubes \cite{saturation-nano, saturation-nanotube} etc..


In this paper, we give an algorithm to count $\Psi(T)$ for any tree $T$.
By using the algorithm we show that for any tree $T$ with $n$ vertices, $\Psi(T)\geq\lceil\frac{n}{2}\rceil$. In addition, the equation holds if and only if $T$ is isomorphic to $S^{0}(K_{1,\frac{n-1}{2}})$ or $S^{2}(K_{1,\frac{n+1}{2}})$ if $n$ is odd, and $T$ is isomorphic to $S^{1}(K_{1,\frac{n}{2}})$ if $n$ is even.

\section{Counting maximal matchings in trees}
Suppose that graph $G$ consists of two connected components $G_1$ and $G_2$. Obviously, any maximal matching of $G$ consists of a maximal matching of $G_1$ and a maximal matching of $G_2$.
So we have $\Psi(G)=\Psi(G_1)\times \Psi(G_2)$.


A connected acyclic graph (i.e., containing no cycles) is called a \emph{tree}.
For a tree, we can choose its any vertex to be the root, and consider it as a root tree.
A tree rooted at vertex $u\in V(T)$ can be denoted by $T_u$. Except for root $u$,
the vertices in $T_u$ with degree equal to one are called \emph{leaves} or \emph{pendant vertices}, the vertices of degree greater than one are called \emph{internal vertices}, and the vertices of degree greater than two are called
\emph{branch vertices}.

In root tree $T_u$, any vertex $x$ and $u$ are connected by exactly one path, denoted by $uTx$. Each vertex $v$ on the path $uTx$ is called an \emph{ancestor} of $x$, and $x$ is called a \emph{successor} of $v$. For two adjacent vertices $i$, $j$ in $T_u$, if $i$ is an ancestor of
$j$, then $i$ is also called a \emph{parent} of $j$, and $j$ is a \emph{child} of $i$.
For any vertex $y$ in $T_u$, all the successors of $y$ and $y$ induce a new tree, we call it a \emph{root subtree} of $T_u$, and denoted by $T_y^{u}$.

\begin{lem}\label{1-degree vertex}
For a tree $T$ with at least two vertices, $u\in V(T)$ is covered by all maximal matchings of $T$ if and only if $u$ is adjacent to one $1$-degree vertex in $T$.
\end{lem}
\begin{proof}
We denote the set of all adjacent vertices of $u$ in $T$ by $N_T(u)$.
If $u\in V(T)$ is covered by all maximal matchings of $T$, then
$u$ is adjacent to $1$-degree vertex in $T$. Suppose not, then $T$ has a matching $M$ which covers all the vertices in $N_T(u)$ and not covers $u$.
$M$ can be extended to a maximal matching $M'$ of $T$, and $M'$ not cover $u$, a contradiction. So $u$ is adjacent to a $1$-degree vertex in $T$.

Suppose that $u$ has a neighbor $x$ which has degree one. Then for any maximal matching $M$ of $T$, $ux\in M$ or $uw\in M$ for some $w\in N_T(u)\setminus \{x\}$. So $u$ must be covered by all maximal matchings of $T$.
\end{proof}
\begin{alg}\emph{(}Counting the number of maximal matchings in trees\emph{)}\label{algorithm}
\begin{description}
\item[Step 1.] Choose a vertex in $T$ to be a root and traverse $T$ to root tree $T_u$.

\item[Step 2.] Give all the leaves $v$ in $T_u$ signs $(\alpha_v, \beta_v, \gamma_v)=(1, 0, 1)$.

\item[Step 3.] If all the signs of vertices in $T_u$ have been determined, then go to Step 5, otherwise, go to Step 4.

\item[Step 4.] If the sign of $v$ has not been determined and the signs of all the children of $v$ have been determined, then sign $v$ as
$(\alpha_v, \beta_v, \gamma_v)$, where (suppose that $x_1, x_2, \ldots, x_k$ are the children nodes of $v$, and their signs are $(\alpha_{x_i}, \beta_{x_i}, \gamma_{x_i})$, $i=1, 2, \ldots, k$.)
$\alpha_v:=\prod\limits_{i=1}^{k}\beta_{x_i}$, $\beta_v:=\sum\limits_{i=1}^{k}\prod\limits_{j=1, j\neq i}^{k}(\alpha_j+\beta_j)\gamma_i$,
$\gamma_v:=\prod\limits_{i=1}^{k}(\alpha_i+\beta_i)$.

\item[Step 5.] Return $\alpha_u+\beta_u$, where $(\alpha_u, \beta_u, \gamma_u)$ are the sign of the root $u$.
\end{description}
\end{alg}

For example, in a path $P_n=v_1v_2\cdots v_n$ with $n$ vertices, we root $P_n$ at the vertex $v_n$. By the Algorithm \ref{algorithm}, the sign of $v_1$ is $(1,0,1)$.
$(\alpha_{v_2}, \beta_{v_2}, \gamma_{v_2})=(0, 1, 1)$, $(\alpha_{v_3}, \beta_{v_3}, \gamma_{v_3})=(1, 1, 2)$, $(\alpha_{v_4}, \beta_{v_4}, \gamma_{v_4})=(1, 2, 2)$, $\ldots$, $(\alpha_{v_n}, \beta_{v_n}, \gamma_{v_n})=(\beta_{v_{n-1}}, \gamma_{v_{n-1}}, \alpha_{v_{n-1}}+\beta_{v_{n-1}})$.

\begin{rem}
In the process of Algorithm \ref{algorithm}, for any root subtree $T_v^{u}$ of $T_u$, we have

\emph{(i)} If the sign of $v$ is $(\alpha_v, \beta_v, \gamma_v)$, then $\Psi(T_v^{u})=\alpha_v+\beta_v$ and $\Psi(T_v^{u}-v)=\gamma_v$.
Moreover, $\alpha_v$ is the number of maximal matchings in $T_v^{u}$ such that $v$ is not covered, and $\beta_v$ is the number of maximal matchings in $T_v^{u}$ such that $v$ is  covered.

\emph{(ii)} If the root $u$ of $T$ is determined in the Algorithm, then the sign of each vertex in $T_u$ is determined. And if change the root of $T$, then each vertex of $T$ will obtain a new sign.
\end{rem}
We replace each leaf of $K_{1,n}$ with a star $K_{1,3}$, and let the center of $K_{1,3}$ is adjacent to the center $u$ of $K_{1,n}$. Then a starlike tree with $4n+1$ vertices is obtained, denoted by
$S_n(K_{1,3}, K_{1,3}, \ldots, K_{1,3})$. Applying the Algorithm \ref{algorithm}, we obtain the following conclusion.
\begin{exa}
$$\Psi(S_n(K_{1,3}, K_{1,3}, \ldots, K_{1,3}))=3^{n}+n\cdot3^{n-1}.$$
\end{exa}
\begin{proof}
We root $S_n(K_{1,3}, K_{1,3}, \ldots, K_{1,3})$ at the center $u$, and denoted by $T_u$. By the Algorithm \ref{algorithm}, all the leaves in $T_u$ have signs $(1,0,1)$, and all the neighbors of $u$ have signs $(0,3,1)$. So
$\alpha_u:=\prod\limits_{x\in N(u)}\beta_{x}=3^{n}$,
and $\beta_u:=\sum\limits_{x\in N(u)}\prod\limits_{y\in N(u)\backslash\{x\}}(\alpha_y+\beta_y)\cdot\gamma_x=n\cdot3^{n-1}$. Hence
$$\Psi(S(K_{1,3}, K_{1,3}, \ldots, K_{1,3}))=\Psi(T_u)=\alpha_u+\beta_u=3^{n}+n\cdot3^{n-1}.$$
\end{proof}

Subdividing $m-t$ ($t<m$) edges of star $K_{1,m}$, we obtain a starlike tree with $2m-t+1$ vertices, denoted by $S^{t}(K_{1,m})$. Clearly, if $t=0$, then $S^{0}(K_{1,m})$ is obtained by subdivide each edge of $K_{1,m}$.
We note that $S^{t}(K_{1,m})$ has only one branch vertex if $m\geq3$.

\begin{lem}\label{S222}
For tree $T$ with $n$ vertices which is isomorphic to $S^{0}(K_{1,\frac{n-1}{2}})$ or $S^{2}(K_{1,\frac{n+1}{2}})$ if $n$ is odd, or $S^{1}(K_{1,\frac{n}{2}})$ if $n$ is even, we have
$\Psi(T)=\lceil\frac{n}{2}\rceil$.
\end{lem}
\begin{proof}
For $n\leq7$, it is easy to check the conclusion. Next, suppose $n\geq8$.
Let $u$ be the only branch vertex in $T$.
If $T$ is isomorphic to $S^{2}(K_{1,\frac{n+1}{2}})$ or $S^{1}(K_{1,\frac{n}{2}})$, then
$u$ is adjacent to a leaf vertex of $T$, and $u$ must be covered by all maximal matchings of $T$ by Lemma \ref{1-degree vertex}.
So $\Psi(T)=d_T(u)=\lceil\frac{n}{2}\rceil$.
If $T$ is isomorphic to $S^{0}(K_{1,\frac{n-1}{2}})$, then
$u$ is not adjacent to a leaf vertex.
So the maximal matchings of $T$ cover $u$ or not, and $\Psi(T)=d_T(u)+1=\lceil\frac{n}{2}\rceil$.
\end{proof}

\begin{lem}
Let $v_1v_2x$ be a path in tree $T$ with $d_T(v_2)=2$ and $v_1$ being a leaf, and $T'$ be a new tree obtained from $T$ by deleting edge $v_1v_2$ and adding edge $v_1x$.
Then $\Psi(T')\geq\Psi(T)$.
\end{lem}
\begin{proof}
Choose $x$ as the root and traverse $T$, $T'$ to root trees $T_x$, $T_x'$ respectively.
Suppose that by the Algorithm \ref{algorithm}, the signs of $x$ in $T_x$ and $T_x'$ be $(\alpha_x, \beta_x, \gamma_x)$ and $(\alpha_x', \beta_x', \gamma_x')$ respectively. Then $\Psi(T')=\alpha_x'+\beta_x'$, and $\Psi(T)=\alpha_x+\beta_x$.
\begin{align}
\alpha_x\!=\prod\limits_{y\in N_T(x)}\beta_{y}=\beta_{v_2}\cdot\prod\limits_{y\in N_T(x)\setminus\{v_2\}}\beta_{y}=\prod\limits_{y\in N_T(x)\setminus\{v_2\}}\beta_{y}\notag
\end{align}
\begin{align}
\alpha_x'\!=\prod\limits_{y\in N_{T'}(x)}\beta_{y}'=\beta_{v_1}'\cdot\beta_{v_2}'\cdot\prod\limits_{y\in N_T(x)\setminus\{v_2\}}\beta_{y}'=0
\end{align}
\begin{align}
\beta_x\!=&\sum\limits_{y\in N_T(x)}\prod\limits_{z\in N_T(x)\setminus\{y\}}(\alpha_z+\beta_z)\cdot\gamma_y\notag\\
\!=&\prod\limits_{y\in N_T(x)\setminus\{v_2\}}(\alpha_y+\beta_y)\cdot\gamma_{v_2}+\sum\limits_{y\in N_T(x)\setminus\{v_2\}}\prod\limits_{z\in N_T(x)\setminus\{y\}}(\alpha_z+\beta_z)\cdot\gamma_y\notag\\
\!=&\prod\limits_{y\in N_T(x)\setminus\{v_2\}}(\alpha_y+\beta_y)+\sum\limits_{y\in N_T(x)\setminus\{v_2\}}\prod\limits_{z\in N_T(x)\setminus\{y, v_2\}}(\alpha_z+\beta_z)\cdot\gamma_y\notag
\end{align}
\begin{align}
\beta_x'\!=&\sum\limits_{y\in N_{T'}(x)}\prod\limits_{z\in N_{T'}(x)\setminus\{y\}}(\alpha_z'+\beta_z')\cdot\gamma_y'\notag\\
\!=&\prod\limits_{y\in N_{T'}(x)\setminus\{v_1\}}(\alpha_y'\!+\!\beta_y')\!\cdot\!\gamma_{v_1}'\!+\!\prod\limits_{y\in N_{T'}(x)\setminus\{v_2\}}(\alpha_y'\!+\!\beta_y')\!\cdot\!\gamma_{v_2}'\!+\!\sum\limits_{y\in N_{T'}(x)\setminus\{v_1,v_2\}}\prod\limits_{z\in N_{T'}(x)\setminus\{y\}}(\alpha_z'\!+\!\beta_z')\!\cdot\!\gamma_y'\notag\\
\!=&2\!\cdot\!\prod\limits_{y\in N_{T'}(x)\setminus\{v_1,v_2\}}(\alpha_y'\!+\!\beta_y')\!+\!(\alpha_{v_1}'\!+\!\beta_{v_1}')\!\cdot\!(\alpha_{v_2}'\!+\!\beta_{v_2}')\!\cdot\!\sum\limits_{y\in N_{T'}(x)\setminus\{v_1,v_2\}}\prod\limits_{z\in N_{T'}(x)\setminus\{y,v_1,v_2\}}(\alpha_z'\!+\!\beta_z')\!\cdot\!\gamma_y'\notag\\
\!=&2\cdot\prod\limits_{y\in N_T(x)\setminus\{v_2\}}(\alpha_y+\beta_y)+\sum\limits_{y\in N_{T'}(x)\setminus\{v_1,v_2\}}\prod\limits_{z\in N_{T'}(x)\setminus\{y,v_1,v_2\}}(\alpha_z'+\beta_z')\cdot\gamma_y'\notag\\
\!=&2\cdot\prod\limits_{y\in N_T(x)\setminus\{v_2\}}(\alpha_y+\beta_y)+\sum\limits_{y\in N_T(x)\setminus\{v_2\}}\prod\limits_{z\in N_T(x)\setminus\{y, v_2\}}(\alpha_z+\beta_z)\cdot\gamma_y\notag
\end{align}
So $\alpha_x'+\beta_x'=\beta_x+\prod\limits_{y\in N_T(x)\setminus\{v_2\}}(\alpha_y+\beta_y)\geq\beta_x+\alpha_x$.
This implies that $\Psi(T')\geq\Psi(T)$.
\end{proof}

\section{The minimum number of maximal matchings in trees}
\begin{lem}\label{ab}
Let $T_u$ be a tree rooted at $u$. Then in the process of Algorithm \ref{algorithm}, for any vertex $x\in V(T_u)$, $\alpha_x\leq\gamma_x$. Moreover, the equation holds if and only if $x$ is a leaf in $T_u$, or for any  child $z\in N_{T_x^{u}}(x)$ of $x$ in $T_u$,  $z$ is adjacent to a leaf of $T_u$.
\end{lem}
\begin{proof}
If $x\neq u$ and $x$ is a leaf in $T_u$, then $\alpha_x=\gamma_x=1$ by the Algorithm \ref{algorithm}.
Next, we suppose that $x$ is not a leaf point in $T_u$.
Recall that $T_x^{u}$ is the sub-tree in $T_u$ with root $x$.
By the Algorithm \ref{algorithm}, $\alpha_x=\prod\limits_{z\in N_{T_x^{u}}(x)}\beta_z$, $\gamma_x=\prod\limits_{z\in N_{T_x^{u}}(x)}(\alpha_z+\beta_z)$.
So $\alpha_x\leq\gamma_x$.
In addition, the equation holds if and only if $\alpha_z=0$ for any child $z$ of $x$ in $T_u$,
that is, $z$ is adjacent to a leaf of $T_u$.
\end{proof}

\begin{lem}
For any leaf point $y$ of a tree $T$ with at least three vertices, $\Psi(T)\geq\Psi(T-y)$. Moreover, suppose the only neighbor of $y$ being $x$,
 the equality holds if and only if except for $y$, any neighbor of $x$ is adjacent to at least one leave vertex.
\end{lem}
\begin{proof}
Let $y$ be the root of tree $T$. We denote $T$ by $T_y$ and suppose that the only adjacent vertex of $y$ is $x$.
$T_x^{y}$ is the sub-tree in $T_y$ with root $x$.
By the Algorithm \ref{algorithm}, the signs of $x$ and $y$ are $(\alpha_x, \beta_x, \gamma_x)$ and $(\alpha_y, \beta_y, \gamma_y)$, and
$\alpha_y=\beta_x, \beta_y=\gamma_x$.
In addition, $\alpha_x\leq\gamma_x$ by Lemma \ref{ab}.
So $\Psi(T)=\alpha_y+\beta_y=\beta_x+\gamma_x\geq\beta_x+\alpha_x=\Psi(T_x^{y})=\Psi(T-y)$, and by Lemma \ref{ab} the equality holds if and only if except for $y$, any neighbor of $x$ is adjacent to at least one leave vertex since $T$ has at least three vertices.
\end{proof}
In fact, the following more general conclusion is true.
\begin{lem}[\cite{Tree-maximal-macthing}]\label{DM}
For any vertex $v$ in a tree $T$, $\Psi(T-v)\leq\Psi(T)$, and the equality holds if and only if $v$ is a leave and except for $v$, any neighbor of $x$ (the only neighbor of $v$) is adjacent to at least one leave vertex.
\end{lem}
By the above Lemma, we have $\alpha_v+\beta_v\geq\gamma_v$ if $T$ is rooted at a vertex $v$.
Furthermore, the following remark holds.
\begin{rem}\label{a+b>r}
Suppose that $T_z^{v}$ is a subtree rooted at $z$ in the root tree $T_v$. Then
$\alpha_z+\beta_z\geq\gamma_z\geq1$ in the intermediate process of the Algorithm \ref{algorithm}.
Moreover, the first equality holds if and only if $z$ has only one neighbor in $T_z^{v}$, say $x$, and
any child $y\in N_{T_z^{v}}(x)\setminus\{z\}$ of $x$ is adjacent to at least one leave vertex in $T$.
\end{rem}

\begin{obs}\label{a+b>0}
For any vertex $x$ of a root tree $T_u$ with at least two vertices, in the intermediate process of the Algorithm \ref{algorithm}, $\beta_x>0$ if $x$ is not a leaf point.
\end{obs}


\begin{thm}\label{even}
If for each leaf point $x$ in a tree $T$ with at least $3$ vertices, $\Psi(T)=\Psi(T-x)$, then the order of $T$ is even.
\end{thm}
\begin{proof}
Suppose that the only adjacent vertex of $x$ is $y$.

\emph{Claim 1.} For any vertex $z\in N_T(y)\setminus \{x\}$, $z$ is adjacent to a leaf point in $T$.

We consider $T$ as the root tree $T_x$. By the Algorithm \ref{algorithm}, $\Psi(T)=\alpha_x+\beta_x$, $\Psi(T-x)=\alpha_y+\beta_y$, and $\alpha_x=\beta_y$, $\beta_x=\gamma_y$. So $\alpha_y=\gamma_y$ since $\Psi(T)=\Psi(T-x)$.
We note that $\alpha_y=\prod\limits_{z\in N_{T_y^{x}}(y)}\beta_z$, and $\gamma_y=\prod\limits_{z\in N_{T_y^{x}}(y)}(\alpha_z+\beta_z)>0$.
By Observation \ref{a+b>0}, $\alpha_z=0$ for any vertex $z\in N_T(y)\setminus\{x\}$. This implies that any vertex $z\in N_T(y)\setminus\{x\}$ is adjacent to a leaf point in $T$.

\emph{Claim 2.} For any vertex $z\in N_T(y)\setminus \{x\}$, the subtree $T_z^{x}$ of $T_x$ has even vertices.

By Claim $1$, the subtree $T_z^{x}$ has at least two vertices.
If for any vertex $z\in N_T(y)\setminus \{x\}$, the subtree $T_z$ is a path with two vertices, then the order of $T$ is even.
Next we show that if the sub-tree $T_z$ is not a path with two vertices, then its order is also even.
By Claim $1$, we know that $z$ is adjacent to a leaf point $w$ in $T$.
Next, we consider the subtree $T_z^{x}$, we rooted it at vertex $w$.
As the discussion of Claim $1$, we know that,
for any vertex $v\in N_{T_z^{x}}(z)\setminus \{w\}$, $v$ is adjacent to a leaf point in $T$. We can recursively prove that $|T_z|$ is even since the order of $T$ is finite.

So the order of $T$ is even.
\end{proof}

In fact, the proof of the Lemma \ref{even} shows the following stronger conclusion.
\begin{cor}
For a tree $T$ with at least $3$ vertices,
if for each leaf point $x$ of $T$, $\Psi(T)=\Psi(T-x)$, then
each inner vertex of $T$ is adjacent to exactly one leaf point.
\end{cor}

\begin{lem}\label{3-leaves}
If $\Psi(T)$ is minimum over all trees with $n$ ($n\geq8$) vertices, then $T$ does not have a vertex which is adjacent to at least three leaves.
\end{lem}
\begin{proof}
By the contrary, we suppose that $T$ has three leaves $x_1, x_2, x_3$ which are adjacent to the common vertex $u$.
For tree $T$, deleting edge $ux_2$ and adding edge $x_1x_2$, we obtain a new tree, denoted by $T'$.
Applying the Algorithm \ref{algorithm}, we choose $u$ as the root in both $T$ and $T'$.
\begin{align}
\Psi(T)\!=&\alpha_u+\beta_u\notag\\
\!=&0+\sum\limits_{y\in N_{T}(u)}\prod\limits_{w\in N_{T}(u)\setminus\{y\}}(\alpha_w+\beta_w)\cdot\gamma_y\notag\\
\!=&3\prod\limits_{w\in N_{T}(u)\setminus\{x_1,x_2,x_3\}}(\alpha_w+\beta_w)+\sum\limits_{y\in N_{T}(u)\setminus\{x_1,x_2,x_3\}}\prod\limits_{w\in N_{T}(u)\setminus\{x_1,x_2,x_3,y\}}(\alpha_w+\beta_w)\cdot\gamma_y \notag
\end{align}
\begin{align}
\Psi(T')\!=&\alpha_u'+\beta_u'\notag\\
\!=&0+\sum\limits_{y\in N_{T'}(u)}\prod\limits_{w\in N_{T'}(u)\setminus\{y\}}(\alpha_w'+\beta_w')\cdot\gamma_y'\notag\\
\!=&2\prod\limits_{w\in N_{T'}(u)\setminus\{x_1,x_3\}}(\alpha_w'+\beta_w')+\sum\limits_{y\in N_{T'}(u)\setminus\{x_1,x_3\}}\prod\limits_{w\in N_{T'}(u)\setminus\{x_1,x_3,y\}}(\alpha_w'+\beta_w')\cdot\gamma_y' \notag\\
\!=&2\prod\limits_{w\in N_{T}(u)\setminus\{x_1,x_2,x_3\}}(\alpha_w+\beta_w)+\sum\limits_{y\in N_{T}(u)\setminus\{x_1,x_2,x_3\}}\prod\limits_{w\in N_{T}(u)\setminus\{x_1,x_2,x_3,y\}}(\alpha_w+\beta_w)\cdot\gamma_y \notag
\end{align}
So $\Psi(T)-\Psi(T')=\prod\limits_{w\in N_{T}(u)\setminus\{x_1,x_2,x_3\}}(\alpha_w+\beta_w)>0$ since $N_{T}(u)\setminus\{x_1,x_2,x_3\}\neq\emptyset$.
This contradicts the choice of $T$.
\end{proof}

By Lemma \ref{3-leaves}, the following conclusion is drawn.
\begin{cor}\label{d-raysstar}
For tree $T$ with $n\geq8$ vertices, if it has a pendent star $K_{1,d}$ with $d\geq4$, then $\Psi(T)$ is not minimum over all trees with $n$ vertices.
\end{cor}

\begin{lem}\label{3-raysstar}
Let $T$ be a tree with $n\geq8$ vertices. If $T$ has a pendent star $K_{1,3}$, then $\Psi(T)$ is not minimum over all trees with $n$ vertices.
\end{lem}
\begin{proof}
We suppose that $T$ has a pendent star $K_{1,3}$, and the center of the star is $x$ with three leaves being $y_1, y_2, z$, where $z$ is not leave in $T$.
$T'$ is a new tree obtained from $T$ by deleting edge $xy_1$ and adding edge $zy_1$. We let the two trees both rooted at $y_1$.
By the Algorithm \ref{algorithm}, we have
$\Psi(T)=\alpha_{y_1}+\beta_{y_1}=\beta_x+\gamma_x=(\alpha_z+\beta_z)+\gamma_z+(\alpha_z+\beta_z)$.
\begin{align}
\Psi(T')\!=&\alpha_{y_1}'+\beta_{y_1}'=\beta_z'+\gamma_z'\notag\\
\!=&\sum\limits_{w\in N_{T'}(z)\setminus\{y_1\}}\prod\limits_{u\in N_{T'}(z)\setminus\{y_1,w\}}(\alpha_u'+\beta_u')\cdot\gamma_w'+\prod\limits_{w\in N_{T'}(z)\setminus\{y_1\}}(\alpha_w'+\beta_w')\notag\\
\!=&\prod\limits_{u\in N_{T'}(z)\setminus\{y_1,x\}}(\alpha_u'+\beta_u')+\sum\limits_{w\in N_{T'}(z)\setminus\{y_1,x\}}\prod\limits_{u\in N_{T'}(z)\setminus\{y_1,x,w\}}(\alpha_u'+\beta_u')\cdot\gamma_w'\notag\\
\!&+\prod\limits_{w\in N_{T'}(z)\setminus\{y_1,x\}}(\alpha_w'+\beta_w')\notag\\
\!=&\prod\limits_{u\in N_{T}(z)\setminus\{x\}}(\alpha_u+\beta_u)+\sum\limits_{w\in N_{T}(z)\setminus\{x\}}\prod\limits_{u\in N_{T}(z)\setminus\{x,w\}}(\alpha_u+\beta_u)\cdot\gamma_w\notag\\
\!&+\prod\limits_{w\in N_{T}(z)\setminus\{x\}}(\alpha_w+\beta_w)\notag\\
\!=&\gamma_z+\beta_z+\gamma_z\notag
\end{align}
So $\Psi(T)-\Psi(T')=(\alpha_z+\beta_z-\gamma_z)+\alpha_z$.
We note that the degree of $z$ is at least two and tree $T\setminus\{y_1,y_2,x\}$ has at least $5$ vertices. If the degree of $z$ is two in $T$, then $\alpha_z>0$, and if the degree of $z$ is at least $3$ in $T$, then $\alpha_z+\beta_z>\gamma_z$ by the Remark \ref{a+b>r}.
To sum up, $\Psi(T)-\Psi(T')>0$, this implies that $\Psi(T)$ is not minimum.
\end{proof}

\begin{lem}\label{Path-4}
For tree $T$ with $n\geq8$ vertices, if it has a pendent path $P_t$ with $t\geq4$, then $\Psi(T)$ is not minimum over all trees with $n$ vertices.
\end{lem}
\begin{proof}
If $\Psi(T)$ is minimum, by Lemma \ref{S222}, $T$ is not isomorphic to a path. So $T$ has one vertex with degree at least $3$.
By the contrary, we suppose that $T$ has a pendent path $P_t$ with $t\geq4$, and let $P_t=x_1x_2x_3\cdots x_t$ with $x_1$ being a leaf in $T$.
We construct a new tree $T'$ by deleting edge $x_1x_2$ and adding edge $x_4x_1$.
Let $x_1$ be the root in both $T$ and $T'$.
By the Algorithm \ref{algorithm}, we have
$\Psi(T)=\alpha_{x_1}+\beta_{x_1}=\alpha_{x_4}+2\beta_{x_4}+\gamma_{x_4}$ and
\begin{align}
\Psi(T')\!=&\alpha_{x_1}'+\beta_{x_1}'=\beta_{x_4}'+\gamma_{x_4}'\notag\\
\!=&\prod\limits_{w\in N_{T'}(x_4)\setminus\{x_3,x_1\}}(\alpha_w'+\beta_w')+\sum\limits_{w\in N_{T'}(x_4)\setminus\{x_1,x_3\}}\prod\limits_{u\in N_{T'}(x_4)\setminus\{x_1,x_3,w\}}(\alpha_u'+\beta_u')\cdot\gamma_w'\notag\\
&+\prod\limits_{w\in N_{T'}(x_4)\setminus\{x_1,x_3\}}(\alpha_w'+\beta_w')\notag\\
\!=&\prod\limits_{w\in N_{T}(x_4)\setminus\{x_3\}}(\alpha_w+\beta_w)+\sum\limits_{w\in N_{T}(x_4)\setminus\{x_3\}}\prod\limits_{u\in N_{T}(x_4)\setminus\{x_3,w\}}(\alpha_u+\beta_u)\cdot\gamma_w\notag\\
&+\prod\limits_{w\in N_{T}(x_4)\setminus\{x_3\}}(\alpha_w+\beta_w)\notag\\
\!=&\gamma_{x_4}+\beta_{x_4}+\gamma_{x_4}\notag
\end{align}
So $\Psi(T)-\Psi(T')=\alpha_{x_4}+\beta_{x_4}-\gamma_{x_4}\geq0$.
If $x_4$ is not a leaf in $T-\{x_1,x_2,x_3\}$, then $\Psi(T)-\Psi(T')>0$.
If $x_4$ is a leaf in $T-\{x_1,x_2,x_3\}$, then $t\geq5$, and we can construct a new tree $T''$ by deleting edge $x_3x_2$ and adding edge $x_2x_5$.
Let $x_1$ be the root in both $T$ and $T''$.
By the Algorithm \ref{algorithm}, we have
$\Psi(T)=\alpha_{x_4}+2\beta_{x_4}+\gamma_{x_4}=\beta_{x_5}+2\gamma_{x_5}+(\alpha_{x_5}+\beta_{x_5})$ and
\begin{align}
\Psi(T'')\!=&\alpha_{x_1}''+\beta_{x_1}''=\gamma_{x_5}''+(\alpha_{x_5}''+\beta_{x_5}'')\notag\\
\!=&\prod\limits_{w\in N_{T''}(x_5)\setminus\{x_2,x_4\}}(\alpha_w''+\beta_w'')+\prod\limits_{w\in N_{T''}(x_5)\setminus\{x_2,x_4\}}\beta_w''
+\prod\limits_{w\in N_{T''}(x_5)\setminus\{x_2,x_4\}}(\alpha_w''+\beta_w'')\notag\\
&+\sum\limits_{w\in N_{T''}(x_5)\setminus\{x_2,x_4\}}\prod\limits_{u\in N_{T''}(x_5)\setminus\{x_2,x_4,w\}}(\alpha_u''+\beta_u'')\cdot\gamma_w''\notag\\
\!=&\prod\limits_{w\in N_{T}(x_5)\setminus\{x_4\}}(\alpha_w+\beta_w)+\prod\limits_{w\in N_{T}(x_5)\setminus\{x_4\}}\beta_w
+\prod\limits_{w\in N_{T}(x_5)\setminus\{x_4\}}(\alpha_w+\beta_w)\notag\\
&+\sum\limits_{w\in N_{T}(x_5)\setminus\{x_4\}}\prod\limits_{u\in N_{T}(x_5)\setminus\{x_4,w\}}(\alpha_u+\beta_u)\cdot\gamma_w\notag\\
\!=&\gamma_{x_5}+\alpha_{x_5}+\gamma_{x_5}+\beta_{x_5}\notag
\end{align}
So $\Psi(T)-\Psi(T'')=\beta_{x_5}>0$ by Observation \ref{a+b>0} and $n\geq8$.
All the above discussions show that $\Psi(T)$ is not minimum.
\end{proof}

\begin{thm}\label{degree-3-num}
For tree $T$ with $n\geq8$ vertices, if $\Psi(T)$ is minimum over all trees with $n$ vertices, then $T$ has exactly one branch vertex.
\end{thm}
\begin{proof}
By Lemma \ref{S222}, $T$ is not a path since $\Psi(T)$ is minimum. So $T$ has at least one branch vertex.
We suppose, on the contrary, that $T$ has at least two branch vertices.
We choose one branch vertex, denoted by $x$, such that $T-x$ has exactly one connected component which contains at least one branch vertex of $T$, denote this component by $G$ and the only neighbor of $x$ in $G$ by $y$.
By Lemmas \ref{3-raysstar}, \ref{Path-4} and Corollary \ref{d-raysstar}, the other connected components of $T-x$ are isomorphic to $P_2$ or $P_1$.
And $T-x$ has at most two components each of which is isomorphic to $P_1$ by Lemma \ref{3-leaves}, has at least one component which is isomorphic to $P_2$ by lemma \ref{3-raysstar}.
Now, we consider the following three cases.

\textbf{Case 1.} $x$ has exactly one neighbor with degree one.

We denote the only leaf vertex which is adjacent to $x$ by $u$. Deleting all the edges $\{xw|w\in N_T(x)\setminus\{u,y\}\}$ and adding edges $\{yw|w\in N_T(x)\setminus\{u,y\}\}$, we obtain a new graph $T'$.
Let $T$ and $T'$ rooted at $u$, by the Algorithm \ref{algorithm}, we have
$\Psi(T)=\alpha_u+\beta_u=\beta_x+\gamma_x=(d_T(x)-2)(\alpha_y+\beta_y)+\gamma_y+(\alpha_y+\beta_y)$ and
\begin{align}
\Psi(T')\!=&\alpha_{u}'+\beta_{u}'=\beta_{x}'+\gamma_x'=\gamma_{y}'+(\alpha_y'+\beta_{y}')\notag\\
\!=&\prod\limits_{w\in N_{T'}(y)\setminus\{x\}}(\alpha_w'+\beta_w')+\prod\limits_{w\in N_{T'}(y)\setminus\{x\}}\beta_w'+\sum\limits_{w\in N_{T'}(y)\setminus\{x\}}\prod\limits_{z\in N_{T'}(y)\setminus\{x,w\}}(\alpha_z'+\beta_z')\cdot\gamma_w'\notag\\
\!=&\prod\limits_{w\in N_{T}(y)\setminus\{x\}}(\alpha_w+\beta_w)+\prod\limits_{w\in N_{T}(y)\setminus\{x\}}\beta_w+(d_T(x)-2)\cdot\prod\limits_{z\in N_{T}(y)\setminus\{x\}}(\alpha_z+\beta_z)\notag\\
&+\sum\limits_{w\in N_{T}(y)\setminus\{x\}}\prod\limits_{z\in N_{T}(y)\setminus\{x,w\}}(\alpha_z+\beta_z)\cdot\gamma_w\notag\\
\!=&\gamma_y+\alpha_y+(d_T(x)-2)\cdot\gamma_y+\beta_y\notag
\end{align}
So $\Psi(T)-\Psi(T')=(\alpha_y+\beta_y-\gamma_y)\cdot(d_T(x)-2)\geq0$.
If $y$ has degree one in the root subtree $T_y^{u}$ of $T$ and (suppose that $z$ is the only neighbor of $y$ in $T_y^{u}$) any neighbor of $z$ in $T_y^{u}-y$ is adjacent to a leaf vertex, then
$\Psi(T)=\Psi(T')$, and
we consider a new tree $T''$ which is obtained by deleting all edges $\{xw|w\in N_T(x)\setminus y\}$ and adding all edges $\{zw|w\in N_T(x)\setminus y\}$. Otherwise, $\Psi(T)-\Psi(T')>0$.
Let the roots of both $T''$ and $T$ be $u$.
We have $\Psi(T)=(d_T(x)-1)(\alpha_y+\beta_y)+\gamma_y=(d_T(x)-1)(\beta_z+\gamma_z)+(\alpha_z+\beta_z)$ and
\begin{align}
\Psi(T'')\!=&\alpha_{u}''+\beta_{u}''=\beta_{z}''+\gamma_z''\notag\\
\!=&\sum\limits_{w\in N_{T''}(z)\setminus\{u\}}\prod\limits_{v\in N_{T''}(z)\setminus\{u,w\}}(\alpha_v''+\beta_v'')\cdot\gamma_w''+\prod\limits_{w\in N_{T''}(z)\setminus\{u\}}(\alpha_w''+\beta_w'')\notag\\
\!=&(d_T(x)-1)\cdot\prod\limits_{w\in N_{T}(z)\setminus\{y\}}(\alpha_w+\beta_w)+\sum\limits_{w\in N_{T}(z)\setminus\{y\}}\prod\limits_{v\in N_{T}(z)\setminus\{y,w\}}(\alpha_v+\beta_v)\cdot\gamma_w\notag\\
&+\prod\limits_{w\in N_{T}(z)\setminus\{y\}}(\alpha_w+\beta_w)\notag\\
\!=&(d_T(x)-1)\cdot\gamma_z+\beta_z+\gamma_z
\end{align}
So $\Psi(T)-\Psi(T'')=(d_T(x)-2)\cdot\beta_z+(\alpha_z+\beta_z-\gamma_z)>0$ since $\beta_z>0$.

\textbf{Case 2.} $x$ has exactly two neighbors each of which has degree one.

Let $u_1$ and $u_2$ be the two neighbors of $x$, each of which has degree one.
Deleting all the edges $\{xw|w\in N_T(x)\setminus\{u_1,y\}\}$ and adding edges $\{yw|w\in N_T(x)\setminus\{u_1,y\}\}$, we obtain a new graph $T^*$.
Let both $T$ and $T^*$ rooted at $u_2$, by the Algorithm \ref{algorithm}, as the proof of Case 1 we have
$\Psi(T)=(d_T(x)-1)\cdot(\alpha_y+\beta_y)+\gamma_y$
and $\Psi(T^*)=(d_T(x)-2)\cdot\gamma_y+\beta_y+\gamma_y$.
So $\Psi(T)-\Psi(T^*)=(d_T(x)-2)\cdot(\alpha_y+\beta_y-\gamma_y)+\alpha_y$.
If $\alpha_y>0$ or $\beta_y>\gamma_y$, then $\Psi(T)>\Psi(T^*)$. If $\alpha_y=0$ and $\beta_y=\gamma_y$, then $y$ is adjacent to a leaf point (say $z$) in $T_y^{u_2}$, and
\begin{align}
\beta_y\!=&\sum\limits_{w\in N_{T}(y)\setminus\{x\}}\prod\limits_{u\in N_{T}(y)\setminus\{x,w\}}(\alpha_u+\beta_u)\cdot\gamma_w\notag\\
\!=&\prod\limits_{u\in N_{T}(y)\setminus\{x,z\}}(\alpha_u+\beta_u)+\sum\limits_{w\in N_{T}(y)\setminus\{x,z\}}\prod\limits_{u\in N_{T}(y)\setminus\{x,w\}}(\alpha_u+\beta_u)\cdot\gamma_w\notag\\
\!=&\gamma_y+\sum\limits_{w\in N_{T}(y)\setminus\{x,z\}}\prod\limits_{u\in N_{T}(y)\setminus\{x,w\}}(\alpha_u+\beta_u)\cdot\gamma_w\notag
\end{align}
So $\sum\limits_{w\in N_{T}(y)\setminus\{x,z\}}\prod\limits_{u\in N_{T}(y)\setminus\{x,w\}}(\alpha_u+\beta_u)\cdot\gamma_w=0$, this contradicts to $N_{T}(y)\setminus\{x,z\}\neq\emptyset$.\\

\textbf{Case 3.} $x$ does not have neighbor with degree one.

Deleting all the edges $\{xw|w\in N_T(x)\setminus\{y\}\}$ and adding edges $\{yw|w\in N_T(x)\setminus\{y\}\}$, we obtain a new tree, denoted by $T'$.
Let both $T$ and $T'$ rooted at $x$, by the Algorithm \ref{algorithm}, as the proof of Case 1 we have
$\Psi(T)=\alpha_x+\beta_x=\beta_y+[d_T(x)-1]\cdot(\alpha_y+\beta_y)+\gamma_y$, and
$\Psi(T')=\alpha_x'+\beta_x'=\beta_y'+\gamma_y'=[d_T(x)-1]\cdot\gamma_y+\beta_y+\gamma_y$.
So $\Psi(T)-\Psi(T')=[d_T(x)-1]\cdot(\alpha_y+\beta_y-\gamma_y)\geq0$.
By Lemma \ref{DM}, if $\Psi(T)=\Psi(T')$, then in the subtree $T_y^{x}$, $y$ is a leave and except for $y$, any neighbor of $z$ (the only neighbor of $y$ in $T_y^{x}$) is adjacent to at least one leave vertex.
In this case, we consider a new tree $T''$ which is obtained from $T$ by deleting all the edges $\{xw|w\in N_T(x)\setminus\{y\}\}$ and adding edges $\{zw|w\in N_T(x)\setminus\{y\}\}$.
Let both $T$ and $T''$ rooted at $x$, by the Algorithm \ref{algorithm}, as the proof of Case 1 we have
$\Psi(T)=\beta_y+[d_T(x)-1]\cdot(\alpha_y+\beta_y)+\gamma_y=[d_T(x)-1]\cdot\beta_z+d_T(x)\cdot\gamma_z+(\alpha_z+\beta_z)$, and
$\Psi(T'')=\alpha_x''+\beta_x''=\beta_y''+\gamma_y''=\gamma_z''+(\alpha_z''+\beta_z'')=\gamma_z+\alpha_z+[d_T(x)-1]\cdot\gamma_z+\beta_z$.
So $\Psi(T)-\Psi(T'')=[d_T(x)-1]\cdot\beta_z>0$.

All the above discussions show that $T$ has exactly one branch vertex.
\end{proof}
Now, we prove the following main result.
\begin{thm}
For any tree $T$ with $n$ vertices, $\Psi(T)\geq\lceil\frac{n}{2}\rceil$. Moreover, the equality holds if and only if $T\cong S^{0}(K_{1,\frac{n-1}{2}})$ or $S^{2}(K_{1,\frac{n+1}{2}})$ if $n$ is odd, $T\cong S^{1}(K_{1,\frac{n}{2}})$ if $n$ is even.
\end{thm}
\begin{proof}
It is easy to check that the theorem holds for $n=2, 3, \ldots, 7$. Next, suppose $n\geq8$.
We prove the theorem by induction on $n$. Suppose that it holds for any tree with order less than $n$.
We choose one leaf point $u$ in $T$, and consider $T$ as the root tree rooted at vertex $u$. Denote the only neighbor of $u$ by $x$.
By Lemma \ref{ab} and Algorithm \ref{algorithm}, $\alpha_x\leq\gamma_x, \alpha_u=\beta_x, \beta_u=\gamma_x$.

If $n$ is even, then $\lceil\frac{n}{2}\rceil=\lceil\frac{n-1}{2}\rceil$.
Since $\Psi(T)=\alpha_u+\beta_u=\beta_x+\gamma_x\geq\beta_x+\alpha_x=\Psi(T-u)$ and by induction $\Psi(T-u)\geq\lceil\frac{n-1}{2}\rceil$, $\Psi(T)\geq\lceil\frac{n}{2}\rceil$.

If $n$ is odd, then $\lceil\frac{n}{2}\rceil=\lceil\frac{n-1}{2}\rceil+1$.
By Theorem \ref{even}, we can choose the leaf point $u$ of $T$ such that $\Psi(T)\geq\Psi(T-u)+1$. By the induction, $\Psi(T-u)\geq\lceil\frac{n-1}{2}\rceil$. So $\Psi(T)\geq\lceil\frac{n}{2}\rceil$.

By Lemma \ref{S222}, we only need to show the necessity.
Since $\Psi(T)$ is minimum over all trees with $n$ vertices, by Theorem \ref{degree-3-num}, $T$ has exactly one vertex of degree at least $3$.
We denote this vertex by $u$.
Applying the Lemmas \ref{3-leaves}, \ref{Path-4} and \ref{3-raysstar}, we have $T\cong S^{0}(K_{1,\frac{n-1}{2}})$ or $S^{2}(K_{1,\frac{n+1}{2}})$ if $n$ is odd, $T\cong S^{1}(K_{1,\frac{n}{2}})$ if $n$ is even.
\end{proof}


%


\end{document}